\documentclass[11pt]{amsart}

\usepackage{amsmath,amssymb,amsfonts,amsthm,amstext,verbatim}
\usepackage{enumerate,color,graphicx,stmaryrd,psfrag} 
\usepackage{url}
\usepackage{mdwlist}   
\usepackage{mathtools} 
\usepackage{wasysym}   
\usepackage{epstopdf} 

\allowdisplaybreaks

\newtheorem{thm}{Theorem}
\newtheorem{lemma}[thm]{Lemma}
\newtheorem{prop}[thm]{Proposition}

\newtheorem{obs}[thm]{Observation}
\newtheorem{assum}{Assumption}

\theoremstyle{remark} 

\theoremstyle{remark} 
\newtheorem*{rem*}{Remark}
\theoremstyle{remark} 

\theoremstyle{remark} 
\newtheorem*{rems*}{Remarks}

\newcommand{\be}[1]{\begin{equation}\label{#1}}
\newcommand{\ee}{\end{equation}}

\newcommand{\ba}[1]{\begin{align}\label{#1}}
\newcommand{\ea}{\end{align}}

\newcommand{\ben}{\begin{equation*}}
\newcommand{\een}{\end{equation*}}

\newcommand{\ban}{\begin{align*}}
\newcommand{\ean}{\end{align*}}

\def\mik{1}
\newcommand\cpsfrag[2]{\ifnum\mik=1\psfrag{#1}{#2}\fi}

\newcommand{\calA}{\mathcal{A}}
\newcommand{\calB}{\mathcal{B}}
\newcommand{\calC}{\mathcal{C}}
\newcommand{\calD}{\mathcal{D}}

\newcommand{\calH}{\mathcal{H}}

\newcommand{\calO}{\mathcal{O}}

\newcommand{\calR}{\mathcal{R}}

\newcommand{\calV}{\mathcal{V}}

\newcommand\La{\Lambda}      

\newcommand{\Ball}{\La}  

\newcommand{\PP}{\mathbb{P}}     
\newcommand{\EE}{\mathbb{E}}     

\newcommand\resp{respectively}

\newcommand{\NN}{\mathbb{N}}     
\newcommand{\ZZ}{\mathbb{Z}}     

\newcommand{\hide}[1]{}

\title[]{Large deviation bounds for the volume of the largest cluster in
2D critical percolation}
\date{\today}

\author[D.~Kiss]{Demeter Kiss}
\thanks{Statistical Laboratory, Centre for Mathematical Sciences,
University of Cambridge and Advanced Institute for Materials Research, Tohoku University, Sendai}

\begin{document}

\begin{abstract}
  Let $M_n$ denote the number of sites in the largest cluster in critical site percolation on the triangular lattice inside a box side length $n$. We give lower and upper bounds on the probability that $M_n / \EE M_n > x$ of the form $\exp(-Cx^{2/\alpha_1})$ for $x \geq 1$ and large $n$ with $\alpha_1 = 5/48$ and $C>0$. Our results extend to other two dimensional lattices and strengthen the previously known exponential upper bound derived by Borgs, Chayes, Kesten and Spencer \cite{BCKS99}. Furthermore, under some general assumptions similar to those in \cite{BCKS99}, we derive a similar upper bound in dimensions $d > 2$. 
\end{abstract}


\maketitle

\section{Introduction and statement of the main results}

For a general introduction to the percolation model we refer to
 \cite{Kesten1982}, \cite{Grimmett1999}, and \cite{Bollobas2006a}. Consider the
critical bond percolation model on the lattice $\mathbb{Z}^d$ for $d\geq 2$. 
For $n\in\mathbb{N}$ let 
\[
  \Ball_n:= \{-n,-n+1,\ldots, n\}^d
\]
denote the hypercube (ball) centred at the origin with radius $n$. For $v\in
V(\mathbb{T})$ we write $\Lambda_n(v) := v + \Lambda_n$. Further let
$\partial A$ denote the (outer) boundary of $A\subseteq\ZZ^d$, that is
\[
  \partial A := \left\{v \in \ZZ^d\setminus A: \exists u\in A \text{ such that } u\sim v\right\}.
\]

We say that two sites $v,w$ are connected by an open path and denote it by
$v\leftrightarrow w$ where there is a sequence of open edges which starts
with $v$, ends with $w$, and the consecutive vertices edges share a vertex. Let
$v\xleftrightarrow{S} w$ denote the event where there is an open path connecting $v$
to $w$ which only uses vertices in $S\subseteq \ZZ^d$. For $A,B\subseteq \ZZ^d$, $A \xleftrightarrow{S} B$ denotes
the event where there are vertices $v\in A, w\in B$ such that $v\xleftrightarrow{S} w$. When $S$ is omitted, it is assumed to
be equal to $\ZZ^d$.

The open cluster of the vertex $v$ in $\Ball_n$ is denoted by
\[
  \calC_n(v) := \left\{ w\in \Ball_n\, | \, w\xleftrightarrow{\Ball_n}
v\right\}.
\]
Herein the size of a cluster is measured by its number of vertices. Further, let
$\calC_n^{(i)}$ denote the $i$th largest cluster in $\Lambda_n$. For $m\leq n$ we write
$\pi(m,n)$ for the probability $\mathbb{P}_{p_c}\left(\partial\Ball_m\leftrightarrow
\partial \Lambda_n \right)$. We set $\pi(n):=\pi(1,n)$. We will work under the following assumptions.

\begin{assum}[Quasi-multiplicativity] \label{assu:q-mult}
  There exists a constant $C_1$ such that for all $0\leq k\leq l\leq m$ we have
  \begin{align}
    \pi(k,l)\pi(l,m) \leq C_{1} \pi(k,m). \label{eq:quasi}
  \end{align}  
\end{assum}

\begin{assum} \label{assu:1arm}
  There exist constants $C_{2}>0$ and $\alpha < d$ such that for all $n\geq m \geq1$
  \begin{align}
    \frac{\pi(n)}{\pi(m)}\geq C_{2}\left(\frac{n}{m}\right)^{-\alpha}. \label{eq:exponent_bound}
  \end{align}  
\end{assum}

Assumption \ref{assu:q-mult} and \ref{assu:1arm} hold for $d = 2$, as proved in \cite{Grimmett1999} and \cite{Nolin2008}.
Furthermore, Assumption \ref{assu:1arm} holds in high ($d\geq 19$) dimensions, however, we do
not expect Assumption \ref{assu:q-mult} to hold in this case. See Remark \ref{rem:high_d}) below for more details on this case. To our knowledge, it is
an open question whether any of Assumption \ref{assu:q-mult} or \ref{assu:1arm} is satisfied in dimensions $3\leq d\leq 18$.

In \cite{BCKS99} the following bound was given:
\begin{thm}[Proposition 6.3 of \cite{BCKS99}]\label{thm:BCKS} Suppose that Assumption \ref{assu:1arm} holds. Then there exist positive constants
$c_1,c_2$ such that for all $x,n\geq 0$,
\begin{align} \label{eq:BCKS}
  \mathbb{P}_{p_c}\left(|\calC_n^{(1)}|\geq xn^d\pi(n)\right) \leq
c_1\exp(-c_2x).
\end{align}
\end{thm}

We strengthen this result when both of Assumption \ref{assu:q-mult} and \ref{assu:1arm} are satisfied:

\begin{thm} \label{thm:main}
  Let $d\geq 2$, and suppose that Assumptions \ref{assu:q-mult} and \ref{assu:1arm} hold.
  There exist positive constants $c_1,c_2$ depending only on $d$ and the constants appearing in the assumptions, such that for all $n,u > 1$,
\begin{align}
  \mathbb{P}_{p_c}\left(|\calC_n^{(1)}|\geq n^d\pi(n/u)\right)
  & \leq c_1\exp(-c_2 u^d).
\end{align}
Furthermore, for $d = 2$ there are constants $c_3,c_4>0$ such that the lower bound
\begin{align}
  \mathbb{P}_{p_c}\left(|\calC_n^{(1)}|\leq n^d\pi(n/u)\right)
  & \geq c_3\exp(-c_4 u^d) \label{eq:main_lbound}
\end{align}
holds for all $1\leq u \leq n$.
\end{thm}

The lower bound in Theorem \ref{thm:main} follows from standard RSW methods, nevertheless, for completeness we include its proof in Section \ref{ssec:pf_main}.
The upper bound above relies on Theorem \ref{thm:contr} below, which is our main contribution. Let
\begin{equation}
  \calV_n := \left\{v \in \Lambda_n\, | \, v\leftrightarrow
\partial \Lambda_{2n}\right\} \label{eq:def_V_n}
\end{equation}
denote the set of vertices in $\Lambda_n$ which are connected to $\partial\Lambda_{2n}$.

\begin{thm} \label{thm:contr} Let $d\geq 2$, and suppose that Assumptions \ref{assu:q-mult} and \ref{assu:1arm} hold.
  There is a constant $c_1$ such that for all $n,u>0$
\begin{equation}
  \mathbb{E}_{p_c} \binom{\left|\calV_n\right|}{k} \leq (c_1
n^d\pi(n/\sqrt[d]{k})/k)^k. 
\end{equation}

Consequently, for some positive constants $c_2,c_3$, we have
\begin{equation}
  \mathbb{P}_{p_c}\left(\left|\calV_n\right|\geq
n^d\pi\left(n/u\right)\right) \leq c_2\exp(-c_3 u^d).
\end{equation}
The constants $c_1,c_2,c_3$ above only depend on $d$ and the constants appearing in Assumptions \ref{assu:q-mult} and \ref{assu:1arm}.
\end{thm}

A weaker version of Theorem \ref{thm:contr} is proved in \cite{BCKS99} as Lemma 6.1. Theorem \ref{thm:main} follows from
Theorem \ref{thm:contr} by arguments analogous to those in \cite{BCKS99} which lead from \cite[Lemma 6.1]{BCKS99}
to Theorem \ref{thm:BCKS}. Thus we only prove Theorem \ref{thm:contr} and the lower bound in Theorem \ref{thm:main} here.

\begin{rems*}
  \begin{enumerate}[i)] 
    \item Our motivation for studying the size of large critical clusters comes from the forest-fire processes described as follows.
      Let $\lambda$ be some small positive number. At time $0$ all the vertices of $\ZZ^d$ are empty. As time goes on, empty vertices get occupied by a tree at rate $1$, independently from each other.
      Vertices with trees get struck by lightning at rate $\lambda$ independently from each other. When a tree gets struck by lightning, its forest (its connected component in $\ZZ^d$ of vertices with trees)
      is ignited, that is, all of the trees are removed in this forest. Then trees occupy empty vertices with rate $1$, and lightnings strike and so on.
	  We are particularly interested in the case where $\lambda > 0$ is small.	  

      As we can see, a forest burns down at rate proportional to its size, thus a precise control of the size of critical clusters can be useful for the study of the processes above.
    \item \cite[Proposition 6.3]{BCKS99} also treats the case where the percolation
      parameter $p$ is different from $p_c$. Our results extend to
      this case in an analogous way as in \cite{BCKS99}. Furthermore, Assumptions \ref{assu:q-mult} and \ref{assu:1arm},
      our results, as well as those in \cite{BCKS99}, in the case $d = 2$ hold for site/bond percolation on other lattices: As long as the lattice 
      is invariant under a translation, a rotation around the origin with some angle and a reflection on one
      of the coordinate axes, the results above follow. Furthermore, these results remain valid for some inhomogeneous percolation models.
      See \cite{Grimmett1999} for more details.
    \item The proof of Theorem \ref{thm:contr} relies on the method presented in \cite{SDP13}. However, the computation there
      only considers the case $d = 2$. As we will see below, the arguments in \cite{SDP13} extend to the case $d \geq 3$ in a straightforward way.

      Furthermore, by the results in \cite{D-CST} the arguments in \cite{SDP13} can also be adapted for
      the critical two dimensional FK percolation model with $q\geq 1$. Hence statements analogous
      to those in Theorem \ref{thm:main} and \ref{thm:contr} remain valid in such context.
    \item Recall a ratio limit theorem, Proposition 4.9 of \cite{Garban2010}
      for the one arm events. Combining it with Theorem \ref{thm:main} we get,  
      for site percolation on the triangular lattice,
      \begin{align*}
    	\mathbb{P}_{p_c}\left(\left|\calC_n^{(1)}\right|\geq x n^2\pi(n)\right)
	    & \leq c_1\exp(-c_2 x^{96/5}), \\
    	& \geq c_3\exp(-c_4 x^{96/5}) 
      \end{align*}
      with some universal constants $c_i$ for all $x>0$ and $n\geq n_0(x)$.
    \item The upper bound in Theorem \ref{thm:main} trivially extends to $|\calC_n^{(l)}|$ the volume of the $l$th largest cluster. Furthermore, in dimension $2$ the same lower bound with different constants also holds. Its derivation is analogous to that for the largest cluster, hence we omit it.      
    \item Theorem \ref{thm:contr} gives upper bounds on the moments and the tail probability of $\calV_n/n^2\pi(n)$, where, roughly speaking,
      $\calV_n$ counts the points in $\Ball_n$ with one long open arm. Similar upper bounds can be achieved for the number of points with multiple
      disjoint arms.

      Let $k\in \NN$ and $\sigma \in \{0,1\}^k$. Let $\pi_\sigma(m,n)$ denote the probability that $\partial \Ball_m$ and
      $\partial \Ball_n$ are connected by $k$ disjoint arms, where in a counter-clockwise order of these arms the $i$th arm is open when $\sigma_i =1$
      and dual closed otherwise. Suppose that Assumption \ref{assu:q-mult} and \ref{assu:1arm} are satisfied when $\pi$ is replaced by $\pi_\sigma$ with some
      constants $C_1,C_2$ and for some $\alpha_\sigma > 0$ not necessarily smaller than $d$. We have two cases: when $\alpha_\sigma < d$, we get results analogous
      to Theorem \ref{thm:contr}. However, when $\alpha_\sigma >d$, by checking the computations in the proof of Theorem \ref{thm:contr}, one gets
      \begin{align*}
	    \EE_{p_c}\binom{|\calV^\sigma_n|}{k} \leq c_1^k n^d \pi_\sigma(n)
      \end{align*}
      for some constant $c_1$ where $\calV_n^\sigma$ denotes the multi-arm analogue of $\calV_n$.

      A lower bound analogous to that in the second part of Theorem \ref{thm:main} hold in two dimensions when $\sigma$ switches colours at most four times
      and $\alpha_\sigma < 2$. However, in this case the construction in the lower bound is more delicate, but we can apply the
      strong separation lemma \cite[Lemma 6.2 and 6.3]{DS11} to deduce the required lower bound. 
    \item \label{rem:high_d} Let us turn to the case $d\geq 19$. Kozma and Nachmias \cite[Theorem 1]{KN11} proved that $\pi(n)=O(n^{-2})$ building
      on the results in \cite{Ha08}. This combined with \cite[Theorem 5]{Ai97} gives that $|\calC_n^{(1)}|$ is of order $n^{4+o(1)}$. Hence the bounds in Theorem \ref{thm:BCKS}
      and \ref{thm:main} are much weaker than those in \cite[Theorem 5]{Ai97}. Nevertheless, we get some new conditional results which
      are interesting in dimensions below $19$.
    \item We note some results on the distribution of $|\calC_n^{(l)}|$ for $l\geq1$. We already mentioned the results of \cite{BCKS99} which are the most relevant for our purposes. The same authors in \cite{BCKS01} describe the connection between the volume and the diameter of the largest critical and near-critical clusters. 
      J\'{a}rai \cite{Ja03} showed, among other things, that the microscopic scale behaviour of the largest critical clusters can be described by that of the incipient infinite cluster.
      Finally, van den Berg and Conijn \cite{vdBC12} proved that the probability of $|\calC_n^{(1)}|/n^2\pi(n)\in (a,b)$ is positive for all $0<a<b$ for sufficiently large $n$. While in \cite{vdBC13} they showed, roughly speaking, that the distribution of $|\calC_n^{(1)}|/n^2\pi(n)$ has no atoms for large $n$ and that $|\calC_n^{(l)}|-|\calC_n^{(l+1)}|= O(n^2\pi(n))$ for $l\geq1$.
    \end{enumerate}
\end{rems*}

\subsection*{Organization of the paper} In Section \ref{sec:NotandPre} we provide some more notation. We sketch the arguments
of \cite{SDP13} which are essential for the proofs of our results in Section \ref{ssec:sdp}. Building on these results, we prove
Theorem \ref{thm:contr} in Section \ref{ssec:pf_contr}. We conclude in Section \ref{ssec:pf_main} where we deduce the lower bound
in Theorem \ref{thm:main}.

\subsection*{Acknowledgements}
The author thanks Ren\'{e} Conijn and Rob van den Berg for fruitful discussions
and for suggesting the problem. He is grateful to Markus Heydenreich for his advice on percolation in high dimensions.

\section{Notation and preliminaries}\label{sec:NotandPre}

The space of configurations is  $\Omega := \{0,1\}^{E(\ZZ^d)}$. For $\omega\in\Omega$ let $\omega(e)\in \{0,1\}$ denote
its value at $e\in E(\ZZ^d)$. We say that $e\in E(\ZZ^d)$ is open, if $\omega(e) = 1$, otherwise $e$ is closed. 
For $p\in[0,1]$ let $\PP_{p}$ denote the product measure on $\Omega$ where $\PP_{p}(\omega(e) = 1) = p$. Let $p_c = p_c(d)$
denote the critical percolation parameter. That is, $p_c = \sup\{p\,|\,\PP_{p}(0\leftrightarrow\infty)\} = 0)$. 


\subsection{The counting argument of \cite{SDP13}} \label{ssec:sdp}

The proof of Theorem \ref{thm:contr} is based on a counting argument found in
\cite{SDP13}. There the argument there is a strengthens the proof of \cite[Lemma 6.1]{BCKS99} and it is used to count certain passage points,
which, roughly speaking, are the starting points six disjoint open and closed arms.
Herein we give a sketch of the argument in the one arm case.

Let $k\in \mathbb{N}$ and 
\[
  X = \left\{x_1,x_2,\ldots, x_{k}\right\} \subseteq \Lambda_n.
\]
We give a bound on the probability of the event $\left\{\calV_n\supseteq
X\right\}$, but first some definitions.

Let $T_0$ denote the empty graph on the vertex set $X$. Let us start blowing a
ball at each point of $X$ at unit speed. That is, at time $t\geq 0,$ we have the
balls $\Lambda_t(x)$, $x\in X$.

For small values of $t$ these balls are pairwise disjoint. As $t$ increases, more and more of these balls intersect each other.
Let $r_1,$ denote the smallest $t$ when the first pair of balls touch. We pick one such pair balls in some deterministic way, with centres
$u_1,v_1\in X$. We draw an edge $e_1$ between $u_1$ and $v_1$ and label it with $l(e_1):= r_1$, and get the graph $T_1$. Note that 
$||u_1 - v_1||_\infty = 2r_1$. Then we continue with the growth process, and stop at time $r_2$ if we find a
pair of vertices $u_2,v_2\in X$ such that $u_2,v_2$ are in different connected
components of $T_1$ and $\Lambda_{r_2}(u_2)$ and $\Lambda_{r_2}(v_2)$ touch.
Then we draw an edge $e_2$ between one such deterministically chosen pair with the
label $l(e_2) := r_2$ and get $T_2$. Note that it can happen that $r_1 = r_2$.
We continue with this procedure till we arrive to the tree $T_{k-1}$.
Let $\calR(X)$ denote the multiset containing $r_i$ for $i=1,2,\ldots,k-1$.

As we saw above, $r_1 = \frac{1}{2}\min_{u,v\in X, u\neq v}||u-v||$.
Furthermore, it is easy to see that for $i = 1,2,\ldots,k-1$ there are at least $k + 1 - i$
vertices of $X$ such that any pair of them is at least $2r_i$ distance from other.
This combined with the pigeon-hole principle provides the following observation:
\begin{obs} \label{obs:num_X}
  For all $i\in [0,\sqrt[d]{k-1}]\cap \ZZ$ we have $r_{k - i^d} < \frac{n}{i}$.
\end{obs}

We say that $B$ is a blob, if $B$ is a non-empty connected component of $T_i$
for some $i$. In the growth process above blobs merge with
other blobs and form bigger ones over time. Let 
\begin{align*}
  b(B) 
  & := \min\{r_i \,:\, B \text{ is a connected component of } T_i\}, \\
  d(B) 
  & := \max\{r_i \,:\, B \text{ is a connected component of } T_i\}
\end{align*}
denote the birth time, and the death time of a blob $B$. It is easy to see that the sets
\newcommand{\ib}{ib}
\newcommand{\ob}{ob}
\[
  G(B) := \begin{cases}
            \bigcup_{x\in B}\Lambda_{d(B)}(x) \setminus \bigcup_{x\in B}\Lambda_{b(B)}(x) & B \neq X, \\
            \Ball_{2n} \setminus \bigcup_{x\in B}\Lambda_{d(B)}(x) & B = X
          \end{cases}
\]
are pairwise disjoint. See Figure \ref{fig:blobs}.
  \begin{figure}
    \begin{center}
      \includegraphics[width=0.5\textwidth]{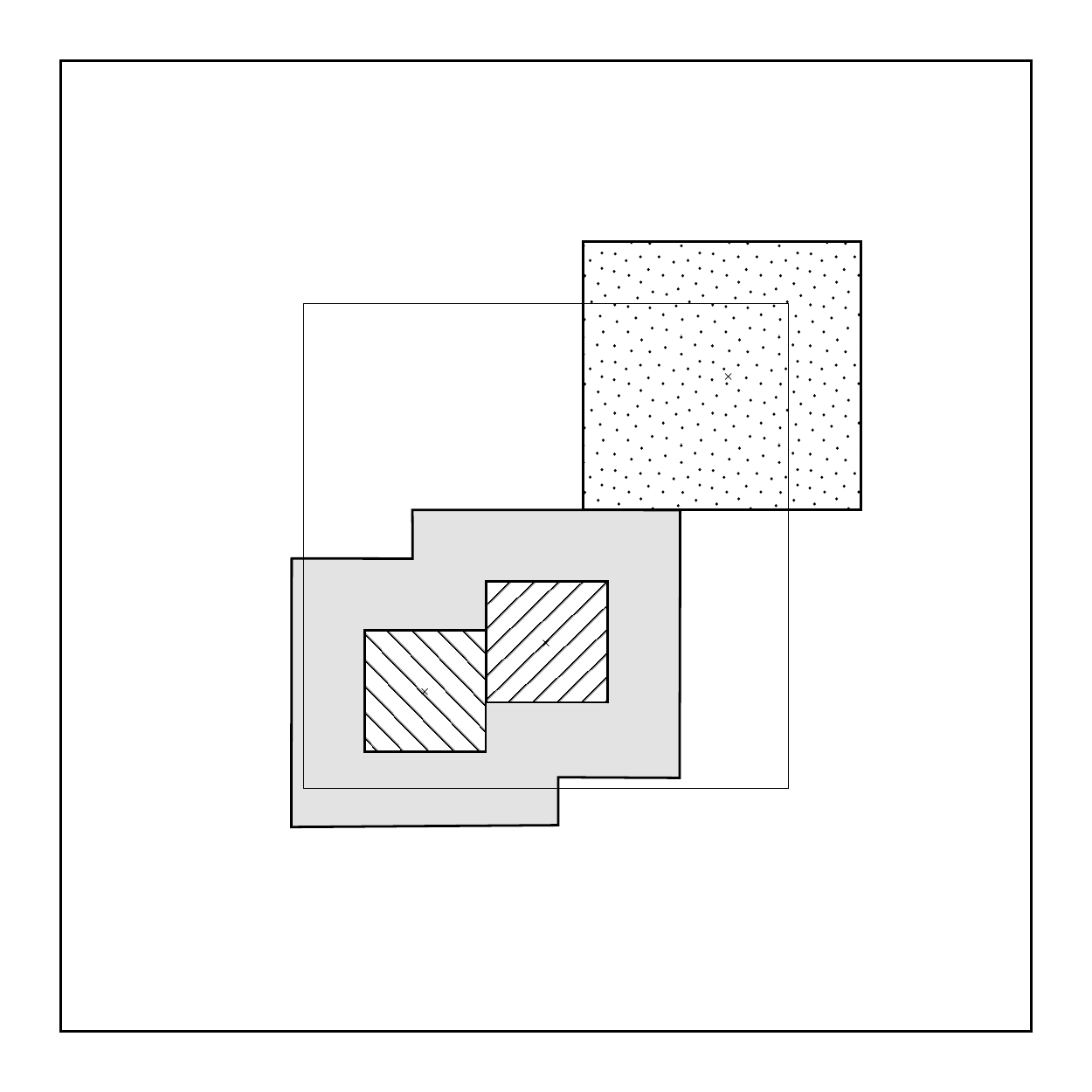}
    \end{center}
    \caption{The areas with different patterns correspond the sets $G(B)$.}
    \label{fig:blobs}
  \end{figure}
Let
\begin{align*}
  \ib(B) &:= \partial \left(\bigcup_{x\in B}\Lambda_{b(B)}(x)\right), 
  & \ob(B)&:=
    \begin{cases}
      \partial \left(\bigcup_{x\in B}\Lambda_{d(B)}(x)\right) & B\neq X,\\
      \partial \Ball_{2n} & B = X
    \end{cases}
\end{align*}
denote the boundary of the inner and outer faces of the sets $G(B)$, \resp.
Now we are ready to make a bound on the probability $\PP(\calV_n\supseteq X)$.
Recall the definition of $\calV_n$ from \eqref{eq:def_V_n}. For all $x \in V(B)$ we have
\[
  \{\calV_n\supseteq X\} \subseteq \{x\leftrightarrow \partial\Lambda_{2n}\} \subseteq \{\ib(B)\leftrightarrow\ob(B)\}.
\]
The events $\{\ib(B)\leftrightarrow\ob(B)\}$ are independent since they depend only
on the state of the edges in $G(B)$, which are pairwise disjoint subsets of $\Ball_{2n}$. Hence 
\begin{align*}
  \PP_{p_c}\left(\calV_n\supseteq X\right)
  & \leq \PP_{p_c}\left(\bigcap_{B \text{blob}}\{\ib(B)\leftrightarrow\ob(B)\}\right) \\
  & \leq \prod_{B \text{ blob}} \PP_{p_c}\left(\ib(B)\leftrightarrow\ob(B)\right).
\end{align*}
Then, as in the proof of \cite[Proposition 14]{SDP13}, an induction on the blobs leads to the following bound.
\begin{prop} \label{prop:prob_X}
  Suppose that Assumption (I) and (II) holds. Then there is a constant $C_3 = C_3(c_1,c_2,\alpha,d)$ such that 
  \[
    \PP_{p_c}\left(\calV_n\supseteq X\right) \leq C_3 \pi(n) \prod_{r\in\calR(X)}C_3\pi(r)
  \]
  for all $X\subseteq \Ball_n$
\end{prop}

Proposition \ref{prop:prob_X} provides an upper bound on
$\PP_{p_c}\left(\calV_n\supseteq X\right)$ as a function of $\calR(X)$. 
To give a bound on $\EE_{p_c}\binom{|\calV_n|}{k},$ we bound the
number of sets $X$ such that $\calR(X) = R$ for fixed $R$. By arguments
analogous to the proof of \cite[Proposition 15]{SDP13} we get the following.
\begin{prop} \label{prop:num_X}
  There is a universal constant $C_4$ such that for all multisets $R$ with $k-1$ elements we have
  \begin{equation}
    \#\left\{X \subseteq \Lambda_n \, : \, |X| = k,\,\calR(X)=R \right\} \leq C_4 \calO(R) n^d \prod_{r\in R} dC_4r^{d-1},
  \end{equation}
  where $\calO(R)$ denotes the number of different ways the elements of $R$ can be ordered.
\end{prop}

\section{Proof of Theorem \ref{thm:main} and \ref{thm:contr}}

We start with the following consequence of Assumption (II).
\begin{lemma}[Lemma 4.4 of \cite{BCKS99}] \label{lemma:sum_pi}
If Assumption (II) holds, then there is a constant $C_{5} = C_{5}(C_{2},\alpha, d)$ such that for all $n\geq0$ we have
\begin{align}
  \sum_{k=1}^n k^{d-1}\pi(k)\leq C_5 n^d\pi(d). \label{eq:sumbound}
\end{align}  
\end{lemma}

\subsection{Proof of Theorem \ref{thm:contr}} \label{ssec:pf_contr}
Combining Proposition
\ref{prop:prob_X} and \ref{prop:num_X} with $C_6 = dC_3C_4 $  we get:
\begin{align}
 \EE\binom{|\calV_n|}{k}
 & = \sum_{X\subseteq \Lambda_n}\PP_{p_c}\left(\calV_n\supseteq X\right)  \nonumber \\
 & \leq d\sum_{R} C_3C_4 \calO(R) n^d\pi(n)\prod_{r\in R} dC_3C_4r^{d-1}\pi(r)
  \label{eq:1st_bound} \\
 & = C^k_6 n^d\pi(n) \sum_{\tilde{R}} \prod_{\tilde{r}\in\tilde{R}} \tilde{r}^{d-1}\pi(\tilde{r}) =  C^k_6 n^d\pi(n) \left(\sum_{r=1}^nr^{d-1}\pi(r)\right)^{k-1}
  \label{eq:2st_bound} 
\end{align}
where the first summation in \eqref{eq:1st_bound} runs over the $k-1$ element
mulitsets of $\{1,2,\ldots,n\}$, while in \eqref{eq:2st_bound} $\tilde{R}$ runs through the $k-1$ long sequences in $\{1,2,\ldots,n\}$.
Note that by Observation \ref{obs:num_X}, many terms in \eqref{eq:2st_bound} are redundant. We exploit this in the following.

Let $\bar{r}_i$ denote the $i$th largest element of $\tilde{R}$. Observation \ref{obs:num_X} provides an upper bound on $\EE\binom{|\calV_n|}{k}$
where in the sum in \eqref{eq:2st_bound} we restrict to the terms such that $\bar{r}_i \leq n / 2^l$ for all $i$ with $2^{dl} \leq i < 2^{d(l+1)}$.
We indicate this restriction by an additional tilde above the sum. Let $j := \lfloor\log_{2^d}(k)\rfloor$ and $m = k-1- 2^{dj}$. We arrive to the
following bound:
\begin{align}
 \EE\binom{|\calV_n|}{k} \leq
 & \; C^k_6 n^d\pi(n) \widetilde{\sum_{\tilde{R}}} \prod_{\tilde{r}\in\tilde{R}} \tilde{r}^{d-1}\pi(\tilde{r}) \nonumber\\
 \leq & \; C_6^k n^{d}\pi(n) \binom{k-1}{2^d-1, (2^d-1)2^d, \ldots,(2^d-1)2^{d(j-1)},m} \label{eq:3rd_bound}\\ 
 & \;\prod_{i=1}^{j-1}\left(\sum_{r=1}^{n/2^{i}}r^{d-1}\pi(r)\right)^{(2^d-1)2^{di}} 
\left(\sum_{r=1}^{n/2^{j-1}}r^{d-1}\pi(r)\right)^{m}.\nonumber 
\end{align}
The multinomial term in \eqref{eq:3rd_bound} bounds the number of ways we can order $k-1$ (not necessarily different) numbers when we do not
distinguish between the largest $2^d-1$, the next $(2^d-1)2^d$ largest,..., and the next $(2^d-1)``^{d(j-1)}$ largest of them.
The product terms in \eqref{eq:3rd_bound} apply the above bounds on the range of $\bar{r}_i$.
Hence by Lemma \ref{lemma:sum_pi}, we have that
\begin{align}
 \EE\binom{|\calV_n|}{k} \leq 
& \; (C_5C_6)^k n^{dk} \binom{k-1}{2^d-1, (2^d-1)2^d,
\ldots,(2^d-1)2^{d(j-1)},m} \nonumber\\
& \; 2^{-m(j-1)d}\prod_{i = 1}^ {j-1} 2^{-di(2^d-1)2^{id}} \cdot
\pi(n)\pi(n/2^{j-1})^m \prod_{i=1}^{j-1}\pi(n/2^i)^{(2^d-1)2^{di}}. \label{eq:5th_bound}
\end{align}

We estimate the multinomial, and the two product terms separately.
It is a simple computation to show that there is a constant $C_7 = C_7(d)$ 
such that
\begin{align}
 \binom{k-1}{2^d-1, (2^d-1)2^d, \ldots,(2^d-1)2^{d(j-1)},m} \leq C_7^{k-1}, \label{eq:6th_bound}
\end{align}
and that
\begin{align}
  2^{-m(j-1)d}\prod_{i = 1}^ {j-1} 2^{-di(2^d-1)2^{id}} \leq C_7^k k^{-k} \label{eq:7th_bound}
\end{align}
for all $k\geq1$.
We combine \eqref{eq:5th_bound}, \eqref{eq:6th_bound}, and \eqref{eq:7th_bound} with the trivial bound $\pi(n/\sqrt[d]{k})^{k}$
for the product of $\pi$'s, and get 
\begin{align}
   \EE\binom{|\calV_n|}{k} \leq C_{8}^k n^{kd}k^{-k}\pi(n/\sqrt[d]{k})^{k} \label{eq:8th_bound}
\end{align}
with $C_8 = C_5C_6C_7^2$. This finishes the proof of the first part of Theorem \ref{thm:contr}.
\medskip

Let us proceed to the proof of the second part. The statement is trivial for $u>n$, hence we assume $u\in[1,n]$
in the following. For $t\geq1$ by \eqref{eq:8th_bound} we get
\begin{align*}
  \EE\left(t^{|\calV_n|}\right)
  & = \sum_{k=1}^\infty (t-1)^k \binom{|\calV_n|}{k} \\
  & \leq \sum_{k=0}^\infty \left((t-1)C_8n^d\pi(n/\sqrt[d]{k})/k \right)^k.
\end{align*}
Take $t = 1 +\frac{u^d}{C_2C_8n^d\pi(n/u)}$ where $u\in[1,n]$. With Assumption (II) we get
\begin{align}
 \EE\left(t^{|\calV_n|}\right)
  & \leq \sum_{k=0}^\infty \left(\frac{u^d\pi(n/\sqrt[d]{k})}{C_2k\pi(n/u)} \right)^k \nonumber\\
  & \leq \sum_{k=0}^{C^{-1}_2u^d} \left(\frac{u^d}{C_2k}\right)^k 
  + \sum_{k=C^{-1}_2u^d+1}^\infty \left(\frac{u^d}{k}\right)^{(1-\alpha/d)k}\nonumber\\
  & \leq \sum_{k=0}^{\infty} \frac{u^{dk}}{C_2^k k!}
  + C^{-1}_2u^d\sum_{l=1}^\infty\left(l^{1-\alpha/d}\right)^{-C^{-1}_2u^dl}\nonumber\\
  & \leq
\exp(C^{-1}_2u^d)+C^{-1}_2u^d\sum_{l=1}^\infty l^{-(1-\alpha/d)l}\nonumber\\
  &\leq C_9 \exp(C^{-1}_2u^d)\label{eq:E(t^V)_upperbound}
\end{align}
for some constant $C_9 = C_9(\alpha,d)$.
Note that the function $x\rightarrow (1+x)^{1/x}$ is decreasing, and that
$\frac{u^d}{n^d\pi(n/u)} \leq C_2^{-1}\left(u/n\right)^{d-\alpha}\leq C_2^{-1}$ since $u\in[1,n]$. Hence there
is a constant $C_{10}$ such that for all $K>0$
\begin{align}
  t^{Kn^d\pi(n/u)}
  & = \left(1+\frac{u^d}{C_2C_8n^d\pi(n/u)}\right)^{Kn^d\pi(n/u)} \geq
\exp\left(C_{10}K u^d\right).\label{eq:t^slowerbound}
\end{align}
Then the Markov inequality, \eqref{eq:E(t^V)_upperbound} and \eqref{eq:t^slowerbound} with $K=2/(C_2C_{10})$ gives that
\begin{align}
   \mathbb{P}_{p_c}\left(\left|\calV_n\right|\geq \frac{2}{C_2C_{10}}n^d\pi\left(n/u\right)\right) \leq
C_9\exp\left(-u^d/C_8\right), \label{eq:almostdone}
\end{align}
From \eqref{eq:almostdone} by Assumption \ref{assu:1arm} the second part of Theorem \ref{thm:contr} follows. This finishes the proof of
Theorem \ref{thm:contr}. $\square$

\subsection{Proof of the lower bound of Theorem \ref{thm:main}} \label{ssec:pf_main}
In this section we consider the case $d=2$.

For $n,m\geq1$
let $B(n,m)$ denote the rectangle $B(n,m) := [0,n]\times [0, m] \cap \ZZ^2$. Further, let $\calH(B(n,m))$ denote the event that there
is an open path connecting $\{0\}\times [0,m]$ to $\{n\}\times[0,m]$. The notation extends to translates of $B(n,m)$ in the usual way.
Furthermore, we define the event $\calV(B(n,m))$ that there is a vertical crossing of $B(n,m)$.
The following well-known statement fist appeared in \cite{Seymour1978}, see also \cite{Russo1981}.
\begin{lemma}[RSW]\label{lemma:RSW}
  There is a positive constant $C_{11} >0$ such that for all $n\geq1$
\begin{align*}
  \PP_{p_c}(\calH(B(n,2n))) \geq e^{-C_{11}}.
\end{align*} 
\end{lemma}
We say that an event $\calA$ is increasing, if $\omega\in\calA$ then $\omega'\in\calA$ for all $\omega'\in\Omega$ with $\omega'\geq \omega$, where $\geq$
is understood coordinate-wise. We recall the FKG -inequality \cite{Fortuin1971}:
\begin{lemma}(FKG)\label{lemma:FKG}
  Let $\calA,\calB$ be increasing events, then 
\begin{align*}
  \PP_{p_c}(\calA\cap \calB) \geq \PP_{p_c}(\calA)\PP_{p_c}(\calB).
\end{align*}
\end{lemma}

We start with the following lemma.
\begin{lemma}\label{lemma:lbound_calV}
  There are positive constants $C_{12},C_{13}$ such that for all $n\geq1$
  \begin{align*}
    \PP_{p_c}(|\calV_n| & \geq C_{12}n^2\pi(n)) \geq e^{-C_{13}}.
  \end{align*}
\end{lemma}
\begin{proof}[Proof of Lemma \ref{lemma:lbound_calV}]
  Simple computation gives that
  \begin{align*}
    \EE_{p_c}(|\calV_n|)\geq n^2\pi(3n)\geq C_2 3^{-\alpha}n^2\pi(n).
  \end{align*}
  This combined with Theorem \ref{thm:contr} provides the desired constants $C_{12}$ and $C_{13}$.
\end{proof}

Now we proceed to the proof of the lower bound in Theorem \ref{thm:main}. 
\begin{proof}[Proof of the lower bound in Theorem \ref{thm:main}]
For $v\in \ZZ^2$, we set $B(v;n,m) := B(n,m)+v$, and 
\begin{align*}
  \calV_n(v) := \left\{w\in \Lambda_n(v)\, | \, w\leftrightarrow \partial \Lambda_{2n}(v)\right\}
\end{align*}

Note that it is enough to prove \eqref{eq:main_lbound} when $u$ is an integer in $[2,n]$. We set $n' = \lfloor n/u\rfloor$.
Let $\calD_n(u)$ denote the event
\begin{align*}
  \calD_n(u) := \bigcap_{v\in \Ball_{u}} \calH\left(B\left(n'v;n',2n'\right)\right)\cap
  \calV\left(B\left(n'v;2n',n'\right)\right).
\end{align*}
It is easy to check that on the event $\calD_n(u)$, all the vertices $w\in \Ball_{n-n'}$
with $w\leftrightarrow \partial\Ball_{2n'}(w)$ belong to the same cluster. In particular, on $\calD_n(u)$ we have
\begin{align} \label{eq:pf_lower_1}
  \sum_{v\in\Ball_{u-1}} \left|\calV_{n'}\left(n'v\right)\right| \leq |\calC_n^{(1)}|.
\end{align}

Lemma \ref{lemma:RSW} and \ref{lemma:FKG} gives that
\begin{align} \label{eq:pf_lower_2}
  \PP_{p_c}(\calD_n(u))\geq e^{-C_{11} 2u^2}.
\end{align}
Combination of \eqref{eq:pf_lower_1}, \eqref{eq:pf_lower_2} and Lemma \ref{lemma:FKG} gives that for $C_{12}>0$ as in Lemma \ref{lemma:lbound_calV} we have
\begin{align}
  \PP_{p_c}\left(|\calC_n|^{(1)}\geq \frac{C_{12}}{2}n^2\pi(n/ \right.&\! u)\Bigg) \nonumber\\
    &\geq \PP_{p_c}\left(\calD_n(u),\sum_{v\in\Ball_{u-1}} \left|\calV_{n'}\left(n'v\right)\right|\geq \frac{C_{12}}{2}n^2\pi(n/u)\right) \nonumber\\
    &\geq e^{-2C_{10} u^2} \PP_{p_c}\left(\sum_{v\in\Ball_{u-1}} \left|\calV_{n'}\left(n'v\right)\right|\geq \frac{C_{12}}{2}n^2\pi(n/u)\right) \label{eq:pf_lbound}\\
    &\geq e^{-2C_{11} u^2} \PP_{p_c}\left(\calV_{n'}\geq C_{11}n'^2\pi(n')\right)^{u^2} \nonumber\\
    &\geq e^{-(2C_{11}+C_{13}) u^2}.\label{eq:pf_lbound_1}
\end{align}
Above we used Lemma \ref{lemma:lbound_calV} in \eqref{eq:pf_lbound} and in \eqref{eq:pf_lbound_1}. Simple application of Assumption
\ref{assu:1arm} finishes the proof of the lower bound of Theorem \ref{thm:main}.
\end{proof}

\bibliographystyle{amsalpha}
\bibliography{myreflist}
 
\end{document}